\documentclass[12pt,BCOR=9.20mm,paper=a4]{scrartcl}
\usepackage{ayv_options}
\newcommand{\Rdd}{\mathbb{R}^{d\times d}}
\newcommand{\Sd}{S_d^+}
\newcommand{\Lebesgue}{\mathscr{L}^1}
\newcommand{\prd}{\mathscr{P}^2(\rd)}
\title{Fokker-Planck-Kolmogorov  inclusions of the mean field type}
\author{Yurii Averboukh}\address{HSE University, Moscow, Russia;\\ & Krasovskii Institute of Mathematics and Mechanics,\\ &  Yekaterinburg, Russia; \\ & Ural Federal University, Yekaterinburg, Russia; } \email{averboukh@gmail.com}

\date{}

\begin{document}
\maketitle
\begin{abstract}
	The paper studies a dynamical system in the Wasserstein space, where the evolution is governed by a Fokker-Planck-Kolmogorov equation with coefficients chosen from a prescribed set that depends at each point on that point and on the current measure. Under the assumption that the multivalued mapping determining the constraints is convex-valued, upper semicontinuous and satisfies certain growth conditions, we prove an existence theorem and establish the compactness of the solution set. Moreover, we study an optimal control problem for such dynamical systems. 
	\msccode{34G25,49J21,35Q84}
\end{abstract}
\section{Introduction}
This paper studies the nonlinear parabolic Fokker-Planck-Kolmogorov (FPK) equation in which the coefficients are not fixed, but are instead chosen from a prescribed set that depends on the space-time point and on the current measure describing the solution. More precisely, let there exists a multivalued mapping \(F\) assigning to a time \(t\), a phase vector \(x\) and a probability measure \(m\) a set of pairs \((\beta,\Sigma)\), where \(\beta\) is a vector, while \(\Sigma\) is a nonnegative symmetric matrix. The problem is to find a flow of probabilities \((m_t)_{t\in [0,T]}\)
satisfying the Fokker–Planck–Kolmogorov equation
\begin{equation}\label{intrd:eq:FPK}
	\partial_t m_t + L_t^* m_t = 0
\end{equation}
in the distributional sense for some generator \(L_t\)
defined by
\begin{equation}\label{intrd:equality:L}
	L_t\phi(x)=\nabla\phi(x)\cdot b(t,x)+\frac{1}{2}\operatorname{tr}\bigl(\nabla^\top\nabla\phi(x)\cdot A(t,x)\bigr)
\end{equation}
such that
\begin{equation}\label{intrd:cond:b_A}
	(b(t,x),A(t,x))\in F(t,x,m_t)\quad\text{for }m_t\text{-a.e. }x\in \mathbb{R}^d\text{ and a.e. }t\in [0,T].
\end{equation}

Define \(V_F(t,m)\) as the set of generators 
\(L\) such that \(L\phi(x)=\nabla\phi(x)\cdot b(x)+\frac{1}{2}\operatorname{tr}(\nabla^\top\nabla\phi(x)\cdot A(x))\) for some \((b(x),A(x))\in F(t,x,m)\). Then, formally, problem~\eqref{intrd:eq:FPK}--\eqref{intrd:cond:b_A} can be rewritten in the form 
\[\partial_tm_t+(V_F(t,m_t)^*m_t)\ni 0.\] Thus, we call  problem~\eqref{intrd:eq:FPK}--\eqref{intrd:cond:b_A} a Fokker-Planck-Kolmogorov inclusion of the mean field type.

The FPK inclusion of the mean field type naturally appears within the study of nonlinear FPK equations with the mean field interaction in the case where the coefficients can be discontinuous. Indeed, if one consider the equation 
\[\partial_tm_t+(L_t[m_t])^*m_t=0\] with 
\[L_t[m]\phi(x)=\nabla\phi(x)\cdot b(t,x,m)+\frac{1}{2}\operatorname{tr}(\nabla^\top\nabla\phi(x)\cdot A(t,x,m))\] and let the coefficients to be discontinuous, the trick coming from~\cite{filippov1988differential} is to define
\[F(t,x,m)\triangleq \bigcap_{\varepsilon>0}\operatorname{co}\big\{(b(t',x',m'),A(t',x',m')):\,(t',x,m')\in B_\varepsilon(t,x,m)\big\},\] where \(B_\varepsilon\) stands for a neighborhood. 

Another source of the FPK inclusions of the mean field type is
a controlled McKean–Vlasov equation with control constraints depending on time, the phase variable and the distribution. More precisely, the problem is to prove an existence theorem for the controlled SDE
\[dX_t=a(t,X_t,\operatorname{Law}(X_t),u_t)dt+\sigma(t,X_t,\operatorname{Law}(X_t),u_t)dW_t,\, u_t\in U(t,X_t,\operatorname{Law}(X_t)).\] Setting
\[F(t,x,m)\triangleq \Big\{\big(a(t,x,m,u),\sigma(t,x,m,u)\sigma^\top(t,x,m,u)\big):\, u\in U(t,x,m)\Big\},\] we arrive at the FPK inclusion of the mean field type considered in this paper. Notice that under rather general conditions on \(a\) and \(\sigma\), given  a solution of the FPK inclusion constructed for the controlled McKean-Vlasov equation \((m_t)_{t\in [0,T]}\), using the Filippov’s implicit function theorem (see~\cite[Theorem 18.17]{AliprantisBorder2006}) and the  superposition principle \cite{Bogachev2021,Trevisan2016}, one can find a feedback control \([0,T]\times\rd\ni (t,x)\mapsto u(t,x)\in U(t,x,m_t)\) and a stochastic process \((X_t)_{t\in [0,T]}\) such that 
\begin{itemize}
	\item the equation \(dX_t=a(t,X_t,\operatorname{Law}(X_t),u_t)dt+\sigma(t,X_t,\operatorname{Law}(t),u(t,X_t))dW_t\) is satisfies in the weak sense;
	\item \(\operatorname{Law}(X_t)=m_t\).
\end{itemize}

The main results of the paper are as follows. First, we show that, if 
\(F\) is upper semicontinuous, convex-valued, and its values satisfy certain growth conditions, then there exists at least one solution of the mean-field FPK inclusion that meets a given initial condition. Moreover, the set of solutions starting from a compact set is compact. Finally, we consider a control problem with dynamical constraints determined by a FPK inclusion of the mean field type. We show that if the running and terminal costs are lower semicontinuous and satisfy certain growth conditions, then the optimal control problem admits a solution, and the value function is lower semicontinuous.

Now let us give a short literature survey. Our paper combines settings from the theory of differential inclusions and FPK equations. The study of differential inclusions began with the papers by Zaremba \cite{zaremba1936paratingent} and Marchaud \cite{marchaud1938bulletin} in the 1930s. Differential inclusions naturally appear in the study of ODEs with discontinuous right-hand side \cite{filippov1988differential} as well as in control problems \cite{wazewski1962generalisation,filippov1962certain}. The main questions studied within differential inclusion theory are the existence of solutions, their quantity and topological properties, as well as viability in  a given set and the approximation of prescribed trajectories \cite{aubin1984differential,kisielewicz1991differential,tolstonogov2000differential,Aubin:2009}. 

A natural extension of the concept of an ODE is a stochastic differential equation. Considering an SDE with multivalued coefficients leads to the notion of stochastic differential inclusions \cite{kisielewicz2013stochastic, aubin1998viability, aubin2000stochastic, gliklikh2010stochastic}.

Additionally, one may assume that the coefficients of the SDE depend on the current distribution (or law) of the unknown process. This type of equation is often called a McKean–Vlasov SDE (or SDE of mean-field type). We refer to the seminal book~\cite{sznitman1991topics} for an introduction to this field and the link between McKean–Vlasov SDEs and 
\(N\)-particle stochastic systems, known as propagation of chaos. Notice that the study of McKean–Vlasov SDEs particularly involves Wasserstein spaces (see \cite{ambrosio, villani2009optimal} for a detailed exposition of these spaces). The concept of the Wasserstein distance originated in optimal transport theory and also plays a crucial role in this paper.

The distribution of a solution to the SDE satisfies the Fokker–Planck–Kolmogorov equation. At the same time, the FPK equation can be considered independently of its SDE origin (see \cite{Bogachev2015} for existence and uniqueness results). The connection between this equation and stochastic processes is established by the Ambrosio–Figalli–Trevisan superposition principle~\cite{Bogachev2021,Trevisan2016}. We also highlight recent progress concerning nonlinear FPK equations~\cite{barbu2023uniqueness,barbu2023nonlinear,GUILLIN20211,MANITA2015199,manita2014nonlinear}.

As noted above, the FPK inclusion arises in particular from control problems for McKean–Vlasov SDEs (see \cite{pham2017dynamic,bensoussan2013mean,djete2022mckean,carmona2018probabilistic} and the references therein). Furthermore, the PDE approach enables one to view the controlled McKean–Vlasov SDE as a controlled system on the Wasserstein space. A particularly interesting case occurs when the volatility coefficient vanishes, in which case the dynamics reduce to the continuity equation. The corresponding differential inclusions, assuming a continuous vector field, were studied in~\cite{Dapice2026,bonnet2020mean,bonnetweill2025viability,bonnet2021differential,bonnet2022viability,bonnetweill2024caratheodory, bonnetweill2024viability}; the discontinuous case was addressed in~\cite{averboukh2022mean,cavagnari2018superposition}.

The rest of the paper is organized as follows. Section~\ref{sect:prel} introduces the general notation. The main results are formulated in the next section. Section~\ref{sect:gauge} is devoted to the construction of the so called moment gauge function, which is the key tool used to prove the existence result for the FPK inclusion and the compactness of the solution set. In Section~\ref{sect:bounds}, we obtain some a priori bounds for the FPK equation. Based on the results of Sections~\ref{sect:gauge} and~\ref{sect:bounds}, we prove the existence theorem for FPK inclusions in Section~\ref{sect:construct}. The compactness of the solutions to the FPK inclusion starting from a compact set is established in the following section. Finally, Section~\ref{sect:control} contains the analysis of the control problem for the FPK inclusion.

\section{Preliminaries}\label{sect:prel}
\begin{itemize}
	\item Let \(T>0\), and let \(\Lebesgue\) denote the Lebesgue measure on \([0,T]\). 
	\item If \(n\) is a natural number, \(X_1,\ldots,X_n\) are some sets, \(i_1,\ldots,i_k\) are indexes in \(\{1,\ldots,n\}\), then \(p^{i_1,\ldots,i_k}\) stands for the projection operator from \(X_1\times\ldots\times X_n\) onto \(X_{i_1}\times\ldots\times X_{i_k}\) defined by the rule: \(\operatorname{p}^{i_1,\ldots,i_k}(x_1,\ldots,x_n)\triangleq (x_{i_1},\ldots,x_{i_k})\).
	\item Given two measurable spaces \((\Omega,\mathcal{F})\), \((\Omega',\mathcal{F}')\), a measure \(m\) on \(\mathcal{F}\) and a mapping \(h:\Omega\rightarrow\Omega'\) that is \(\mathcal{F}/\mathcal{F}'\)-measurable, we denote by \(h\sharp m\) a push-forward measure on \(\mathcal{F}'\) defined by the rule: for each \(\Upsilon\in\mathcal{F}'\), 
	\[(h\sharp m)(\Upsilon)\triangleq m(h^{-1}(\Upsilon)).\] 
	\item Given a Polish space \((X,\rho_X)\), a closed set \(\Upsilon\subset X\) and a point \(x\in X\), we denote by \(\operatorname{dist}(x,\Upsilon)\) the distance between \(x\) and \(\Upsilon\): \(\operatorname{dist}(x,\Upsilon)\triangleq \inf\{\rho_X(x,y):\, y\in\Upsilon\}\). Moreover, \(\mathbb{B}_R(x)\) is used for the closed ball of  radius \(R\) centered at the point \(x\). If \(X\) is a Banach space, and \(x\) is the origin, we omit the argument.
	\item If \((X,\rho_X)\) is a Polish space, then \(\mathcal{M}(X)\) is used for the space of Borel finite nonnegative measures on \(X\). The space \(\mathcal{M}(X)\) is endowed with the topology of narrow convergence. In what follows, we consider only Borel measures. If \(m\in\mathcal{M}(X)\), then \(\operatorname{supp}(m)\) denotes the support of \(m\), i.e., the minimal closed set such that \(m\) gives zero weight to its complement. Furthermore, \(\mathcal{P}(X)\) stands for the space of  probability measures, i.e., 
	\[\mathcal{P}(X)\triangleq \big\{m\in\mathcal{M}(X):\, m(X)=1\big\}.\] The set \(\mathcal{P}(X)\) inherits the topology of the narrow convergence and is closed in this topology. 
	\item If \(X,Y\) are some Polish spaces, \(\nu\) is a measure on \(Y\),  while \((m_y)_{y\in Y}\) is a weakly measurable family of probability measures on \(X\), then one can define \(\mu\triangleq \nu\star(m_t)_{t\in [s,r]}\in\mathcal{M}(Y\times X)\) by the rule: for every \(\varphi\in C^0_b(Y\times X)\), 
	\begin{equation}\label{prel:intro:integral}\int_{Y\times X}\phi(y,x)\mu(d(y,x))\triangleq \int_{Y}\int_X\varphi(y,x)m_t(dx)\,\nu(dy).\end{equation} Conversely, the disintegration theorem \cite[Theorem 5.3.1]{ambrosio} implies that, given \(\mu\in\mathcal{M}(Y\times X)\) such that \(\operatorname{p}^1\sharp\mu=\nu\), one can construct a weakly measurable family of probability measures \((m_y)_{y\in Y}\) on \(X\) such that~\eqref{prel:intro:integral} holds. This family will be  denoted by \((\mu(\cdot|y))_{y\in Y}\). It is defined in the unique way up to the set of \(\nu\)-measure zero.
	\item The set of all Borel probability measures on \(X\) with finite second moment is denoted by \(\mathscr{P}^2(X)\), i.e., \(\mathscr{P}^2(X)\) consists of all measures \(m\in\mathcal{P}(X)\) such that, for some \(x_*\in X\),
	\[\int_X(\rho_X(x,x_*))^2m(dx)<\infty.\]
	\item The space \(\mathscr{P}^2(X)\) is endowed with the second Wasserstein metric defined by
	\[W_2(m_1,m_2)\triangleq \Bigg[\inf\bigg\{\int_{X\times X}(\rho_X(x_1,x_2))^2\pi(d(x_1,x_2)):\, \pi\in\Pi(m_1,m_2)\bigg\}\Bigg]^{1/2},\] where \(\Pi(m_1,m_2)\) stands for the space of couplings between \(m_1\) and \(m_2\), i.e., the set of all Borel probability measures \(\pi\) on \(X\times X\) such that, for every Borel set \(\Upsilon\subset X\), \(\pi(\Upsilon\times X)=m_1(\Upsilon)\) and \(\pi(X\times\Upsilon)=m_2(\Upsilon)\). 
	
	It is well known \cite[Remark 7.1.7]{ambrosio} that \((\prd,W_2)\) is a Polish space.
	\item We assume that \(\rd\) consists of column-vectors, while \(\rds\) is the set of row-vectors. The norms on \(\rd\) and \(\rds\) are denoted by \(|\cdot|\). If \(x\in\rd\) and \(p\in\rds\), then \(px\) denotes their product. We use \(\Rdd\) to denote the space of \(d\times d\) real matrices. The symbol \({}^\top\) denotes transposition. Moreover, \(\Sd\) stands for the set of all symmetric nonnegative matrices.
	\item \(C_b^k(\rd)\) denotes the space of \(k\)-times differentiable functions from \(\rd\) to \(\mathbb{R}\) whose derivatives up to order \(k\) are bounded. 
	\item If \(s,r\) are real numbers with \(s<r\), then \(C^{1,2}_c((s,r)\times\rd)\) denotes the set of compactly supported real valued functions defined on \((s,r)\times\rd\)  that are continuously differentiable in time and twice continuously differentiable in the phase variable.
	\item If a real-valued function \(\phi\) that is defined on some neighborhood of a point \(y\in\rd\) is differentiable at this point, then \(\nabla\phi(y)\) denotes its derivative. We assume that \(\nabla\phi(y)\) is a row-vector. If \(\phi\) is twice differentiable at \(y\), then \(\nabla^\top\nabla\phi(y)\) denotes the Hessian at this point; \(\nabla^\top\nabla\phi(y)\in \mathbb{R}^{d\times d}\).
\end{itemize}

\section{Main results}\label{sect:main}
As we mentioned in the Introduction, we study the Fokker-Planck-Kolmogorov inclusion of the mean field type, defined as follows.

We assume that  a multivalued mapping 
\(F:[0,T]\times\rd\times\prd\rightrightarrows \rd\times\Sd\) is given. 
It generates the multivalued mapping \(V_F\) that assigns to each 
\((t,m)\in [0,T]\times\prd\) the set of generators (i.e., operators 
from \(C_b^2(\rd)\) to \(\mathbb{R}\)) \(L\) such that 
\(L\in V_F(t,m)\) if, for every \(\phi\in C_b^2(\rd)\), there exists 
\((b,A):\rd\rightarrow\rd\times\Sd\) satisfying 
\[
L\phi(x)=\nabla\phi(x)b(x)+\frac{1}{2}\operatorname{tr}(\nabla^\top\nabla\phi(x)A(x)) 
\]
and
\[
(b(x),A(x))\in F(t,x,m)\quad\text{for } m\text{-a.e. }x\in\rd.
\]
We consider the following FPK inclusion of the mean field type
\begin{equation}\label{main:incl:FPK}
	\partial_t m_t+(V_F(t,m_t))^*m_t\ni 0,
\end{equation}
endowed with the initial condition 
\begin{equation*}\label{main:inicond}
	m_0=\nu.
\end{equation*}
\begin{definition}\label{main:def:solution} We say that a flow of probabilities \((m_t)_{t\in [0,T]}\) is a solution of~\eqref{main:incl:FPK} if there exist Borel functions \(b:[0,T]\times\rd\rightarrow\rd\) and \(A:[0,T]\times\rd\rightarrow\Sd\) such that 
	\begin{itemize}
		\item \((b(t,x),A(t,x))\in F(t,x,m_t)\) for \(m_t\)-a.e. \(x\in\rd\) and a.e. \(t\in [0,T]\);
		\item for every \(\varphi\in C^{1,2}_c((0,T)\times\rd)\) we have 
		\[\int_0^T\int_{\rd}\Big[\partial_t\varphi(t,x)+\nabla\varphi(t,x)b(t,x)+\frac{1}{2}\operatorname{tr}(\nabla^\top\nabla\varphi(t,x)\cdot A(t,x))\Big]m_t(dx)\,dt=0.\]
	\end{itemize}
\end{definition} 
Below, we use the following notation:
\begin{itemize}
	\item For \(\nu\in\prd\), \(\mathcal{S}_F(\nu)\) denotes the set of flows of probabilities \((m_t)_{t\in [0,T]}\) solving~\eqref{main:incl:FPK} and satisfying the initial condition \(m_0=\nu\).
	\item \(\mathcal{E}_F(\nu)\) stands for the set of triples \(((m_t)_{t\in [0,T]},b,A)\) satisfying the conditions of Definition~\ref{main:def:solution} and such that \(m_0=\nu\).
	\item If \(\mathcal{K}\subset \prd\), then \(\mathcal{S}_F(\mathcal{K})\) denotes the set of all flows of probabilities solving~\eqref{main:incl:FPK} and starting from a point in \(\mathcal{K}\), i.e., 
	\[\mathcal{S}_F(\mathcal{K})=\bigcup_{\nu\in\mathcal{K}}\mathcal{S}_F(\nu).\]
\end{itemize} Obviously, \(\mathcal{S}_F(\nu)\) is the projection of \(\mathcal{E}_F(\nu)\), meaning that, if \(((m_t)_{t\in [0,T]},b,A)\in\mathcal{E}_F(\nu)\), then \((m_t)_{t\in [0,T]}\in\mathcal{S}_F(\nu)\).

We study the FPK inclusion under the following hypotheses.
\begin{hypothesis}\label{main:hyp:F_values} The multivalued mapping \(F\) is upper semicontinuous and has nonempty convex values. Moreover, there exist nonnegative constants \(C_b\) and \(C_A\) such that for every \((\beta,\Sigma)\in F(t,x,m)\) 
	\[|\beta|\leq C_b(1+|x|+\mathcal{M}_2(m)),\ \ \|\Sigma\|_F\leq C_A(1+|x|^2+\mathcal{M}_2^2(m)).\] Here \(\mathcal{M}_2(m)\) denotes the second moment of the measure \(m\), i.e., 
	\[\mathcal{M}_2(m)\triangleq \Bigg[\int_{\rd}|x|^2m(dx)\Bigg]^{1/2},\] and \(\|\cdot\|_F\) stands for the Frobenius norm. 
\end{hypothesis} 

We now formulate our main results on the existence and compactness of the solution set; they are proved in Sections~\ref{sect:construct} and~\ref{sect:limit}, respectively.
\begin{theorem}\label{main:th:existence} For every \(\nu\in\prd\), \(\mathcal{S}_F(\nu)\) is a nonempty subset of \(C([0,T];\prd)\).
\end{theorem}
\begin{theorem}\label{main:th:compact} If \(\mathcal{K}\subset\prd\) is compact, then \(\mathcal{S}_F(\mathcal{K})\) is compact in \(C([0,T];\prd)\).
\end{theorem}

The last part of the main results concerns the following control problem for FPK inclusions of the mean field type:
\begin{equation}\label{main:payoff:intro}
	\begin{split}
		\text{minimize } J[(m_t)_{t\in [0,T]}&,b,A]\\ \triangleq G(&m_T)+\int_0^T\int_{\rd} f(t,x,m_t,b(t,x),A(t,x))m_t(dx)\,dt\end{split}
\end{equation} over the set \(\mathcal{E}_F(\nu)\).

We define
\[\begin{split}
	\operatorname{Val}(\nu)\triangleq \Bigg\{J[(&m_t)_{t\in [0,T]},b,A]:\ \ ((m_t)_{t\in [0,T]},b,A)\in\mathcal{E}_F(\nu)\Bigg\}.\end{split}\]

The main result in this part is the following.
\begin{theorem}\label{main:th:control} Assume that \(F\) satisfies Hypothesis~\ref{main:hyp:F_values}. Additionally, suppose that 
	\begin{itemize}
		\item \(G\) and \(f\) are lower semicontinuous;
		\item \(G\) is bounded below on every compact subset of \(\prd\); 
		\item \(f\) is convex with respect to \(\beta\) and \(\Sigma\);
		\item \(f\) satisfies the growth condition \[|f(t,x,m,\beta,\Sigma)|\leq C_f(m)(1+|x|^2+|\beta|^2+\|\Sigma\|_F),\] where \(C_f(m)\) is bounded on each compact subset of \(\prd\).
	\end{itemize}	
	Then, there exists an admissible control process \(((m_t^0)_{t\in [0,T]},b^0,A^0)\) such that 
	\[\operatorname{Val}(\nu)=J[(m_t^0)_{t\in [0,T]},b^0,A^0].\] Moreover, the function \(\operatorname{Val}\) is lower semicontinuous. 
\end{theorem}

\section{Moment gauge function}\label{sect:gauge}
This section slightly enhances the construction of the moment gauge function proposed in \cite{Averboukh_Kolpakova}. It is used to show that, under certain growth conditions, the solutions of the FPK equation starting from a compact subset of \(\prd\) lie in a compact set. This fact plays a crucial role in the subsequent proofs. 

For \(x\in\rd\), we put 
\begin{equation}\label{gauge:intro:a}
	\mathscr{a}(x)\triangleq \big(1+|x|^2\big)^{1/2}.
\end{equation}
\begin{proposition}\label{gauge:prop:g} Let \(\mathcal{K}\subset \prd\) be compact. Then, there exists a function \(g\in C^2((0,+\infty))\) such that 
	\begin{enumerate}[label=(G\arabic*)]
		\item\label{gauge:cond_g:values} \(g(r)>0\), \(g'(r)>0\), \(g''(r)\leq 0\) on \([1,+\infty)\); 
		\item\label{gauge:cond_g:infinity} \(g(r)\rightarrow \infty\) as \(r\rightarrow \infty\);
		\item\label{gauge:cond_g:derivative} \(g'(r)r\leq g(r)\) on \([1,+\infty)\)
		\item\label{gauge:cond_g:int} \[\sup_{m\in\mathcal{K}}\int_{\rd}g(\mathscr{a}(x))\mathscr{a}^2(x)m(dx)<\infty.\]
	\end{enumerate}
\end{proposition} 
\begin{remark} We refer to a function satisfying conditions \ref{gauge:cond_g:values}--\ref{gauge:cond_g:int} as a moment gauge function for the compact set \(\mathcal{K}\).
\end{remark}
\begin{proof}[Proof of Proposition~\ref{gauge:prop:g}] First, we construct an auxiliary function \(g_1\). 
	
	Since \(\mathcal{K}\) is compact in \(\prd\), all measures \(m\in\mathcal{K}\) have uniformly integrable second moments \cite[Proposition 7.1.5]{ambrosio}. Hence, given \(\varepsilon>0\), there exists \(R>0\) such that, for all \(m\in\mathcal{K}\), 
	\[\int_{\mathscr{a}(x)\geq R}\mathscr{a}^2(x)m(dx)<\varepsilon.\] Now we define a sequence \(\{R_n\}_{n=0}^\infty\) as follows. Set \(R_0\triangleq 0\). If \(R_0,\ldots,R_{n-1}\) have already been constructed, choose \(R_n\) such that
	\begin{itemize}
		\item \[\int_{\mathscr{a}(x)\geq R_n}\mathscr{a}^2(x)m(dx)\leq \frac{1}{(n+4)2^n};\]
		\item  
		\(R_n-R_{n-1}\geq 1\vee (R_{n-1}-R_{n-2})\) (here we formally assume \(R_{-1}\triangleq 0\));
		\item \(R_n\geq (2n+2)R_{n-1}(2n-1)^{-1}\).
	\end{itemize} Define the function \(g_1:(-\infty,+\infty)\rightarrow \mathbb{R}\) by the rule: for \(r\in [R_{n-1},R_n]\),
	\[g_1(r)\triangleq \frac{r-R_{n-1}}{R_n-R_{n-1}}+n+2;\] while, on \((-\infty,0)\), \(g(r)\equiv 3\). 
	
	Let us list the properties of \(g_1\).
	\begin{enumerate}[label=(\roman*)]
		\item \(g_1(r)\geq 3\);
		\item \(g_1\) is piecewise differentiable; for \(r\in (R_{n-1},R_n)\),
		\(g'_1(r)=(R_n-R_{n-1})^{-1}\); 
		\item\label{gauge:aux_cond:lim} \(g_1(r)\rightarrow \infty\) as \(r\rightarrow \infty\);
		\item The second distributional derivative of \(g_1\) is equal to \[g''_1(r)=\sum_{n=1}^\infty \big[(R_{n+1}-R_n)^{-1}-(R_{n}-R_{n-1})^{-1}\big]\delta(r-R_n);\] here \(\delta(\cdot)\) denotes the Dirac delta function;
		\item\label{gauge:aux_cond:derivative} \((g_1'(r)+1/2)r\leq g_1(r)-1\). Indeed, we have \((2n-1)R_n\geq 2(n+1)R_{n-1}\). Thus, for each \(r\in [R_{n-1},R_n]\), 
		\[\begin{split}
			\Bigg(g_1'(r)+\frac{1}{2}\Bigg)r&=\frac{3r}{2(R_n-R_{n-1})}\leq \frac{3R_n}{2(R_n-R_{n-1})}\\&\leq n+1\leq \frac{r-R_{n-1}}{R_n-R_{n-1}}+n+1\leq g_1(r)-1.
		\end{split} \]
		\item\label{gauge:aux_cond:int} The quantity \(\int_{\rd}g_1(\mathscr{a}(x))\mathscr{a}^2(x)m(dx)\) is  bounded by \(2+R_1^2\) for each \(m\in\mathcal{K}\). To show this notice that
		\[\begin{split}
			\int_{\rd}g_1(\mathscr{a}(x))\mathscr{a}^2(x&)m(dx)\\=\int_{\mathscr{a}(x)<R_1}&g_1(\mathscr{a}(x))\mathscr{a}^2(x)m(dx)\\&+\sum_{n=1}^\infty\int_{\mathscr{a}(x)\in [R_n,R_{n+1})}g_1(\mathscr{a}(x))\mathscr{a}^2(x)m(dx). \end{split}\] Since \(g_1(r)\leq n+4\) on \([R_n,R_{n+1}]\), we have 
		\[\begin{split}
			\int_{\rd}g_1(\mathscr{a}(x))\mathscr{a}^2(x&)m(dx)\\\leq (1+R_1^2)&+\sum_{n=1}^\infty\int_{\mathscr{a}(x)\in [R_n,R_{n+1})}(n+4)\mathscr{a}^2(x)m(dx) \\ \leq (1+R_1^2)&+\sum_{n=1}^\infty(n+4)\int_{\mathscr{a}(x)\geq  R_n}\mathscr{a}^2(x)m(dx).
		\end{split}\] Recalling that \((n+4)\int_{\mathscr{a}(x)\geq  R_n}\mathscr{a}^2(x)m(dx)\leq 2^{-n}\), we get 
		\[
		\int_{\rd}g_1(\mathscr{a}(x))\mathscr{a}^2(x)m(dx)\leq (2+R_1^2).
		\]
	\end{enumerate} 
	
	To construct the function \(g\), consider a smooth mollifier \(\eta:\mathbb{R}\rightarrow\mathbb{R}\) supported on \([-1,1]\). Define 
	\[g(r)\triangleq (g_1*\eta)(r)=\int_{-1}^{1}g_1(r-r')\eta(r')dr'=\int_{-\infty}^{+\infty}g_1(r')\eta(r-r')dr'.\] Since \(0\leq g'_1(r)\leq 1\), we have \(|g'(r)-g_1'(r)|\leq 1\). Moreover, direct calculations give that, for \(r\in [1,+\infty)\), \(g'(r)>0\) and
	\(|g'(r)- g_1(r)|\leq 1/2\). Similarly, \(g''(r)\leq 0\). Thus property~\ref{gauge:cond_g:values} holds. 
	
	Property~\ref{gauge:cond_g:infinity} follows directly from~\ref{gauge:aux_cond:lim}. 
	
	To show~\ref{gauge:cond_g:derivative}, we use~\ref{gauge:aux_cond:derivative} and the inequalities \(g'(r)\leq g_1'(r)+1/2\), \(g_1(r)-1\leq g(r)\). 
	
	Finally, from~\ref{gauge:aux_cond:int}, we have 
	\[\begin{split}\int_{\rd}g(\mathscr{a}(x))\mathscr{a}^2(x)m(dx)&\leq \int_{\rd}g_1(\mathscr{a}(x))\mathscr{a}^2(x)m(dx)+\int_{\rd}\mathscr{a}^2(x)m(dx)\\&\leq (2+R_1^2)+\int_{\rd}\mathscr{a}^2(x)m(dx).\end{split}\] Since \(\int_{\rd}\mathscr{a}^2(x)m(dx)\) is uniformly bounded on \(\mathcal{K}\), we obtain~\ref{gauge:cond_g:int}.
\end{proof}

\section{FPK equations and a priori bounds}\label{sect:bounds} 
\subsection{Properties of solutions of the FPK equation}
In this section, we consider  the FPK equation on \([s,r]\)
\begin{equation}\label{bound:eq:FPK}\partial_t m_t+L^*_tm_t=0,\end{equation} where 
\[L_t\phi(x)\triangleq \nabla\phi(x)\cdot b(t,x)+\frac{1}{2}\operatorname{tr}(\nabla^\top\nabla \phi(t,x)\cdot A(t,x)),\] while
\(b:[s,r]\times\rd\rightarrow\rd\) and \(A:[s,r]\times\rd\rightarrow\Sd\) satisfy 
\begin{enumerate}[label=(L\arabic*)]
	\item\label{bound:cond_L:Borel} \(b\) and \(A\) are Borel;
	\item\label{bound:cond_L:b} \(|b(t,x)|\leq C^0_b\mathscr{a}(x)\);
	\item\label{bound:cond_L:A} \(A\) is nonnegative and \(\|A(t,x)\|_F\leq C^0_A\mathscr{a}^2(x)\).
\end{enumerate} 
\begin{lemma}\label{bounds:lm:Lyapunov} Let \(b\) and \(A\) satisfy conditions~\ref{bound:cond_L:Borel}--\ref{bound:cond_L:A}. Additionally, suppose that the function \(g\in C^2([1,+\infty))\) satisfies conditions~\ref{gauge:cond_g:values}--\ref{gauge:cond_g:derivative}. Then, for the function \(V_g(x)\triangleq g(\mathscr{a}(x))\mathscr{a}^2(x)\), the following inequality holds true: 
	\[L_tV_g(x)\leq C_1 V_g(x),\] where \(C_1\) is a constant depending only on the dimension \(d\) and the constants \(C^0_b\), \(C^0_A\).
\end{lemma}
\begin{proof} First, we compute
	\[
	\nabla V_g(x)= \big[g'(\mathscr{a}(x))\mathscr{a}^2(x)+2g(\mathscr{a}(x))\mathscr{a}(x)\big] \nabla\mathscr{a}(x);
	\]
	\[
	\begin{split}
		\nabla^\top\nabla V_g(x)=\big[g''(\mathscr{a}(x))\mathscr{a}^2(x)+4g'(\mathscr{a}(x))\mathscr{a}(x)+g(\mathscr{a}(x))\big] \nabla^\top\mathscr{a}(x) \nabla\mathscr{a}(x)&{}\\+\big[g'(\mathscr{a}(x))\mathscr{a}^2(x)+2g(\mathscr{a}(x))\mathscr{a}(x)\big] \nabla^\top\nabla\mathscr{a}&(x),
	\end{split}
	\] where 
	\[\nabla\mathscr{a}(x)=\mathscr{a}^{-1}(x)\cdot x^\top;\]
	\[\nabla^\top \nabla \mathscr{a}(x)=\mathscr{a}^{-3}(x)\mathcal{R}(x),\] with the matrix \(\mathcal{R}(x)=(\mathcal{R}_{i,j})_{i,j=1}^d\) such that \(\mathcal{R}_{i,i}(x)=1+\sum_{k\neq i}x_k^2\), and \(\mathcal{R}_{i,j}(x)=-x_ix_j\) for \(i\neq j\). Now recall~\ref{gauge:cond_g:values}. It gives
	\begin{equation}\label{bound:ineq:V_g_b}\nabla V_g(x)\cdot b(t,x)\leq 3C^0_b V_g(x).\end{equation} Furthermore, since \(\|\mathcal{R}(x)\|_F\leq C'_1\mathscr{a}^2(x)\) for some constant \(C'_1\) depending only on \(d\), \(\nabla^\top\mathscr{a}(x)\nabla\mathscr{a}(x)\geq 0\), and \(g''(r)\leq 0\), we have due to~\ref{gauge:cond_g:values} the estimate:
	\[\begin{split}
		\operatorname{tr}(\nabla^\top\nabla V_g(x) A(t&,x))\\\leq \big[4g'(\mathscr{a}(x))&\mathscr{a}(x)+g(\mathscr{a}(x))\big]\|A(t,x)\|_F\\ + C'_1\big[&g'(\mathscr{a}(x))\mathscr{a}^2(x)+2g(\mathscr{a}(x))\mathscr{a}(x)\big]\mathscr{a}^{-1}(x)\|A(t,x)\|_F. \end{split}\] Hence, condition~\ref{bound:cond_L:A} and the fact that \(g'(r)r\leq g(r)\) yield 
	\[\operatorname{tr}(\nabla^\top\nabla V_g(x) A(t,x))\leq C'_2 V_g(x).\] Combining this with~\eqref{bound:ineq:V_g_b}, we obtain the statement of the lemma.
\end{proof}

We recall the following standard definition from~\cite[Definition 6.1.1]{Bogachev2015}.

\begin{definition}\label{bounds:def:solution} We say that the flow of probabilities \((m_t)_{t\in [s,r]}\) is a solution of~\eqref{bound:eq:FPK} if, for every \(\varphi\in C^{1,2}_c([s,r]\times\rd)\),
	\[\int_s^r\int_{\rd}\big[\partial_t\varphi(t,x)+L_t\varphi(t,x)\big]m_t(dx)\,dt=0.\]
\end{definition} 
The following estimate follows directly from Lemma~\ref{bounds:lm:Lyapunov} and \cite[Theorem 7.1.1]{Bogachev2015}.
\begin{lemma}\label{bound:lm:growth}  Suppose that 
	\begin{itemize}
		\item \(g\in C^2([1,+\infty))\) satisfies conditions~\ref{gauge:cond_g:values}--\ref{gauge:cond_g:derivative};
		\item \(V_g(x)\triangleq g(\mathscr{a}(x))\mathscr{a}^2(x)\);
		\item \(\nu\in\prd\) is such that \(\int_{\rd}V_g(x)\nu(dx)\leq \hat{C}\);
		\item \((m_t)_{t\in [s,r]}\) solves~\eqref{bound:eq:FPK} for some coefficients \(b\) and \(A\) satisfying~\ref{bound:cond_L:Borel}--\ref{bound:cond_L:A}.
	\end{itemize} Then, there exists a constant \(C_2\) depending only on the dimension \(d\) and the constants \(C^0_b\), \(C^0_A\), \(\hat{C}\) such that, for every \(t\in [s,r]\),
	\[\int_{\rd}V_g(x)m_t(dx)\leq C_2.\]
	
\end{lemma}

\begin{lemma}\label{bound:lm:lip} Let \(g\), \(\nu\), \((m_t)_{t\in [s,r]}\) and \(\hat{C}\) be as in the previous lemma. Then, there exists a constant \(C_3\) depending only on \(d\), \(C^0_b\), \(C^0_A\), \(\hat{C}\) such that, for all \(t_1,t_2\in [s,r]\),
	\[W_2^2(m_{t_1},m_{t_2})\leq C_3|t_1-t_2|.\]
\end{lemma}
\begin{proof} First note that the mapping \(t\mapsto m_t\) is narrowly continuous. Thus,  the Ambrosio–Figalli–Trevisan superposition principle (see~\cite[Theorem 1.1]{Bogachev2021}, \cite[Theorem 2.5]{Trevisan2016}) gives that there exists a probability measure \(P\) on \(C([s,r];\rd)\) that solves the martingale problem for the operator \(L\), while \(e_t\sharp P=m_t\) for the evaluation map \(e_t:C([s,r];\rd)\rightarrow \rd\) defined by \(e_t(x(\cdot))\triangleq x(t)\).
	Furthermore, consider the SDE
	\begin{equation}\label{bound:eq:SDE}dX_t=b(t,X_t)+\sigma(t,X_t)dW_t,\end{equation} where \(\sigma^\top(t,x)\sigma(t,x)=A(t,x)\). The existence of a solution to the martingale problem for \(L\) implies that~\eqref{bound:eq:SDE} admits a weak solution with \(\operatorname{Law}(X_t)=m_t\) (see \cite[Proposition 5.4.11]{Karatzas1991}). Since 
	\[\mathbb{E}|X_t|^2=\int_{\rd}|x|^2m_t(dx)\] is uniformly bounded, standard SDE estimates imply the conclusion of the lemma.
\end{proof}

\subsection{Compactness of solutions to the FPK equations}\label{section:compact}
Let \(s,r\in [0,T]\), \(s<r\), and let \(\{\nu^n\}_{n=1}^\infty\subset\prd\), \(\nu\in\prd\) be such that \(W_2(\nu^n,\nu)\rightarrow 0\) as \(n\rightarrow \infty\). Assume that \((m_t^n)_{t\in [s,r]}\) satisfies the FPK equations
\begin{equation*}\partial_t m_t^n+(L^n_t)^*m_t^n=0\end{equation*} on \([s,r]\) and the initial condition~\(m^n_s=\nu^n\). Here,
\[L_t^n\phi(x)\triangleq \nabla\phi(x)\cdot b^n(t,x)+\frac{1}{2}\operatorname{tr}(\nabla^\top\nabla \phi(t,x)\, A^n(t,x)),\] and
\(b^n:[s,r]\times\rd\rightarrow\rd\), \(A^n:[s,r]\times\rd\rightarrow\Sd\) satisfy conditions~\ref{bound:cond_L:Borel}--\ref{bound:cond_L:A} with the same constants \(C^0_b\), \(C^0_A\).

\begin{lemma}\label{comapct:lm:limit} There exists a subsequence \(\{n_k\}_{k=1}^\infty\) and a flow of probabilities \((m_t)_{t\in [s,r]}\) such that, for \(m_t^{(k)}\triangleq m_t^{n_k}\), we have 
	\[\sup_{t\in [s,r]}W_2\big(m_t^{(k)},m_t\big)\rightarrow 0\quad\text{as }k\rightarrow\infty.\]
\end{lemma}
\begin{proof}
	First, let \(\mathcal{K}\triangleq \{\nu^n\}_{n=1}^\infty\cup \{\nu\}\). Since \(\mathcal{K}\) is compact, Lemma~\ref{bound:lm:growth} gives \(\int_{\rd}V_g(x)m_t^n(dx)\leq C_2\), where \(g\) is a moment gauge function for \(\mathcal{K}\). Moreover, from~\cite[Lemma 5.1.7]{ambrosio} it directly follows that the set \(\{m:\, \int_{\rd}V_g(x)m(dx)\leq C_2\}\) is compact. In particular, \(\int_{\rd}|x|^2m^n_t(dx)\) is uniformly bounded. Lemma~\ref{bound:lm:lip} implies that the flows \(\{(m_t^n)_{t\in [s,r]}\}_{n=1}^\infty\) are equicontinuous. By the Arzelà–Ascoli theorem, we obtain the desired statement.
\end{proof}

\begin{lemma}\label{compact:lm:equivalence}
	Let a sequence of flows \(\{(m_t^n)_{t\in [s,r]}\}_{n=1}^\infty\) satisfy 
	\begin{itemize}
		\item \(m_t^n\in\mathcal{K}\) for each \(n\) and \(t\in [s,r]\), where \(\mathcal{K}\) is a compact set independent of \(t\) and \(n\);
		\item \(\sup_{t\in [s,r]}W_2(m^n_t,m_t^\infty)\rightarrow 0\) for some flow \((m_t^\infty)_{t\in [s,r]}\).
	\end{itemize} Define, for \(n\in\mathbb{N}\cup\{\infty\}\), \(\mu^n\triangleq \Lebesgue\star(m_t^n)_{t\in [s,r]}\). Then \(\{\mu^n\}_{n=1}^\infty\) converges narrowly to \(\mu^\infty\).
\end{lemma}
\begin{proof}
	First, note that \(m_t^\infty\in\mathcal{K}\) for every \(t\in [s,r]\).  
	
	Then, we will use the truncation arguments borrowed from~\cite[\S~7.1]{Bogachev2015}. Let \(\varphi\in C_b([s,r]\times\rd)\). For each \(\varepsilon>0\) there exists \(R>0\) such that 
	\[\Bigg|\int_{[s,r]\times \rd}\varphi(t,x)\mu^n(d(t,x))-\int_{[s,r]\times \rd}\zeta_R(x)\varphi(t,x)\mu^n(d(t,x))\Bigg|\leq \varepsilon\] for every \(n\in\mathbb{N}\cup\{\infty\}\), where \(\zeta:\rd\rightarrow [0,1]\) is 1-Lipschitz, \(\zeta_R(x)=1\) for \(|x|\leq R\) and \(\zeta_R(x)=0\) for \(|x|\geq R+1\). This follows from the fact that every compact subset of \(\prd\) has uniformly integrable second moments.
	
	Thus, it suffices to consider a function \(\varphi\in C_b([s,r]\times\rd)\) that is bounded and equicontinuous w.r.t. the phase variable. Let    \(\varsigma(\cdot)\) denote modulus of continuity w.r.t. the phase variable. For \(k\in\mathbb{N}\), define
	\[\varphi_k(t,x)\triangleq \inf_{x'\in\rd}\big[\varphi(t,x')+k|x-x'|\big].\] Then, \(\varphi_k\) is \(k\)-Lipschitz in \(x\), and
	\[\|\varphi_k-\varphi\|\leq \varsigma(2\|\varphi\|k^{-1}),\] where \(\|\varphi\|\triangleq \sup_{(t,x) \in [s,r]\times\rd}|\varphi(t,x)|\). Choosing \(k\) to be large, we find that it remains to prove
	\begin{equation}\label{compact:conv:mu}\int_{[s,r]\times \rd}\varphi(t,x)\mu^n(d(t,x))\rightarrow \int_{[s,r]\times \rd}\varphi(t,x)\mu^\infty(d(t,x))\end{equation} for a Lipschitz function \(\varphi\). 
	
	By the definition of \(\mu^n\), 
	\[\begin{split}
		\Bigg|\int_{[s,r]\times \rd}\varphi&(t,x)\mu^n(d(t,x))-\int_{[s,r]\times \rd}\varphi(t,x)\mu^\infty(d(t,x))\Bigg| \\&\leq
		\int_s^r\bigg|\int_{\rd}\varphi(t,x)m^n_t(dx)-\int_{\rd}\varphi(t,x)m^\infty_t(dx)\bigg|dt.
	\end{split}\]
	
	For each \(t\in [s,r]\) and natural \(n\), let \(\pi_t^n\) be an optimal coupling between \(m_t^n\) and \(m_t^\infty\). If \(C'_\varphi\) is a common Lipschitz constant of \(\varphi\) in the phase variable, then
	\[
	\begin{split}
		\Bigg|	\int_{\rd}\varphi(t,x)m^n_t(dx)-\int_{\rd}\varphi(t,x)m^\infty_t(dx)\Bigg| &\leq C'_\varphi(\varphi)\int_{\rd\times\rd}|x-y|\pi_t^n(d(x,y))\\&\leq C'_\varphi(\varphi)W_2(m_t^n,m_t^\infty).
	\end{split}
	\] Since \(\sup_{t\in [s,r]}W_2(m_t^n,m_t^\infty)\rightarrow 0\), we obtain~\eqref{compact:conv:mu} for every Lipschitz \(\varphi\).
\end{proof}

\section{Construction of a solution to the FPK inclusion}\label{sect:construct}
Let \(\nu\in\prd\), and let \(n\) be a natural number. Put \(t^n_i\triangleq iT/n\) for \(i=0,\ldots,n\). The flow of probabilities \((m_t^n)_{t\in [0,T]}\) is constructed stepwise. First, set \(m_0^n\triangleq \nu\). If \(m^n_t\) has been constructed for \(t\in [0,t^n_i]\), then, by the Kuratowski–Ryll-Nardzewski selection theorem \cite[Theorem 18.13]{AliprantisBorder2006}, there exist Borel functions \(b^n_i:\rd\rightarrow \rd\), \(A^n_i:\rd\rightarrow\Sd\) such that 
\[(b^n_i(x),A^n_i(x))\in F(t^n_i,x,m_{t^n_i}^n) \quad\text{for }m^n_{t^n_i}\text{-a.e. }x\in\rd.\] We define \(m^n_t\) for \(t\in [t^n_i,t^n_{i+1}]\) as the solution of the FPK equation on \([t^n_i,t^n_{i+1}]\) with the generator \(L^n_{i,t}\) given by
\begin{equation*}\label{construct:intro:L_n}
	L^n_{i,t}\phi(x)\triangleq \nabla\phi(x)\cdot b^n_i(x)+\frac{1}{2}\operatorname{tr}\big((A^n_i(x)+n^{-1}\operatorname{I}_d)\nabla^\top\nabla\phi(x)\big).
\end{equation*} Here \(\operatorname{I}_d\) denotes  the \(d\times d\) identity matrix.

Thanks to \cite[Theorem 6.6.2]{Bogachev2015}, \(m^n_t\) exists on each interval \([t_i^n,t_{i+1}^n]\) and hence on the whole \([0,T]\). For simplicity, for \(t\in [t_i^n,t_{i+1}^n)\), denote
\[b^n(t,x)\triangleq b_i^n(x),\qquad A^n(t,x)\triangleq A_i^n(x).\] Moreover, set
\[L^n_t\phi(x)\triangleq \nabla\phi(x)\cdot b^n(t,x)+\frac{1}{2}\operatorname{tr}\big((A^n(t,x)+n^{-1}\operatorname{I}_d)\nabla^\top\nabla\phi(x)\big).\] By construction,  the FPK equation
\[\partial_t m_t^n+(L_t^n)^*m_t^n=0\] is satisfied in the distributional sense; moreover, \(m_0^n=\nu.\)

\begin{lemma}\label{construct:lm:M_2} There exists a constant \(C_3\) such that, for every natural \(n\) and every \(t\in [0,T]\),
	\[\mathcal{M}_2(m^n_t)\leq C_3.\]
\end{lemma}
\begin{proof}
	Indeed, given \(i=0,\ldots,n-1\), for the Lyapunov function \(V_2(x)\triangleq |x|^2\) we have
	\begin{equation}\label{construct:ineq:L_n}
		\begin{split}
			L^n_{i}V_2(x)&\leq 2C_b|x|\big(1+|x|+\mathcal{M}_2(m^n_{t^n_i})\big)+2C_A\big(1+|x|^2+\mathcal{M}_2^2(m^n_{t^n_i})\big) \\  &\leq C'_3(1+\mathcal{M}_2(m^n_{t^n_i}))+C'_4|x|^2
		\end{split}
	\end{equation} where \(C'_3,C'_4\) are constants depending only on \(C_b\) and \(C_A\).
	
	Using truncation arguments as in \cite[Theorem 7.1.1]{Bogachev2015} and Gronwall's inequality, we obtain, for \(t\in [t_i^n,t_{i+1}^n]\),
	\[\int_{\rd}|x|^2m_t^n(dx)\leq \Bigg[\int_{\rd}|x|^2m_{t_i^n}^n(dx)+C'_3(t-t_i^n)(1+\mathcal{M}_2^2(m_{t_i^n}^n))\Bigg]e^{C'_4(t-t_i^n)}. \] Thus, for \(t\in [t_i^n,t_{i+1}^n]\),
	\[\mathcal{M}_2^2(m^n(t))\leq \exp((C'_3+C'_4)(t-t_i^n))\big[\mathcal{M}_2^2(m^n_{t_i^n})+(t-t_i^n)\big].\] Summing these inequalities gives the uniform bound
	\[\mathcal{M}_2^2(m^n(t))\leq \exp((C'_3+C'_4)(t-t_i^n))\big[\mathcal{M}_2^2(\nu)+t\big].\] This proves the lemma.
\end{proof}
The following statement directly follows from Lemma~\ref{construct:lm:M_2} and Hypothesis~\ref{main:hyp:F_values}.
\begin{corollary}\label{construct:corollary:cond}
	Each generator \(L^n_\cdot\) satisfies conditions~\ref{bound:cond_L:Borel}--\ref{bound:cond_L:A} with constants \(C_b^0\) and \(C_A^0\) independent of \(n\).
\end{corollary}
\begin{proof}[Proof of Theorem~\ref{main:th:existence}]
	By Corollary~\ref{construct:corollary:cond} and Lemma~\ref{comapct:lm:limit}, there exists a subsequence \(\{n_k\}_{k=1}^\infty\) such that, setting \(m_t^{(k)}\triangleq m_t^{n_k}\), we have \(W_2(m^{(k)}_t,m_t)\rightarrow 0\) uniformly in \(t\in [0,T]\) for some flow of probabilities \((m_t)_{t\in [0,T]}\). We shall prove that \((m_t)_{t\in [0,T]}\) solves  FPK inclusion~\eqref{main:incl:FPK}. To this end, define 
	\begin{itemize}
		\item \(\mu^{(k)}\triangleq \Lebesgue\star (m_t^{(k)})_{t\in [0,T]}\);
		\item \(\mu\triangleq \Lebesgue\star (m_t)_{t\in [0,T]}\);
		\item \(\chi^{(k)}\triangleq (\operatorname{Id},b^{n_k},A^{n_k})\sharp\mu^{(k)}\).
	\end{itemize} By Lemma~\ref{compact:lm:equivalence}, \(\{\mu^{(k)}\}_{k=1}^\infty\) converges narrowly to \(\mu\). Moreover, by the construction of \(\chi^{(k)}\), for every \(\varphi\in C_c^{1,2}((0,T)\times\rd)\),
	\begin{equation}\label{construct:equality:chi_k}\begin{split}
			\int_{[0,T]\times\rd\times\rd\times \Sd}\bigg[\partial_t&\varphi(t,x)+\nabla\varphi(t,x)\cdot\beta \\&+\frac{1}{2}\operatorname{tr}\big(\nabla^\top\nabla\varphi(t,x)(\Sigma+(n^k)^{-1}\operatorname{I}_d)\big)\bigg]\chi^{(k)}(d(t,x,\beta,\Sigma))=0.\end{split}\end{equation}
	Furthermore, let \(\theta^n(t)\triangleq \sup\{t_i^n:\, t_i^n\leq t\}\). Then, 
	\begin{equation}\label{construct:incl:chi_k}\operatorname{dist}((\beta,\Sigma);F(\theta^{n_k}(t),x,m^{(k)}_{\theta^n(t)}))=0\end{equation} for \(\chi^{(k)}\)-a.e. \((t,x,\beta,\Sigma)\).
	
	For each natural \(l\) and \((t,x,m)\in [0,T]\times\rd\times\prd\), define  
	\[F^{[l]}(t,x,m)\triangleq \overline{\operatorname{co}}\bigcup_{t',m':|t-t'|\leq 2^{-l},\, W_2(m,m')\leq 2^{-l}} F(t',x,m').\] By Hypothesis~\ref{main:hyp:F_values},
	\begin{equation}\label{construct:equality:F_limit}F(t,x,m)=\bigcap_{l=1}^\infty F^{[l]}(t,x,m).\end{equation} Recall that \(t-(n^{k})^{-1}T\leq \theta^{n_k}(t)\leq t\), \(\{m^{(k)}_t\}_{k=1}^\infty\) converges uniformly to \(m_t\), while  \((m^{(k)})_{t\in [0,T]}\) is uniformly Hölder continuous (see Lemma~\ref{bound:lm:lip}). Thus, for each \(l\), there exists \(k_0\) such that for all \(k\geq k_0\),
	\[F(\theta^{n_k}(t),x,m^{(k)}_{\theta^n(t)})\subset F^{[l]}(t,x,m_t).\] Together with~\eqref{construct:incl:chi_k}, this gives
	\begin{equation}\label{construct:equality:F_l_chi_k}
		\int_{[0,T]\times\rd\times\rd\times \Sd}\operatorname{dist}((\beta,\Sigma),F^{[l]}(t,x,m_t))\,\chi^{(k)}(d(t,x,\beta,\Sigma))=0.
	\end{equation}
	
	We claim that \(\{\chi^{(k)}\}_{k=1}^\infty\) is tight. Due to the construction of the measures~\(\chi^{(k)}\), for \(\mu^{(k)}\)-a.e. \((t,x)\in [0,T]\times\rd\), 
	\begin{equation}\label{construct:incl:chi_k_tx}\begin{split}
		\operatorname{supp}\big(&\chi^{(k)}(\cdot,\cdot|t,x)\big)\\&\subset 
		\big\{(\beta,\Sigma)\in \rd\times\Sd:\, |\beta|\leq C_b^0\mathscr{a}(x),\, \|\Sigma\|_F\leq C_A^0\mathscr{a}^2(x)\big\}
		.\end{split}\end{equation} Recall that \((\chi^{(k)}(\cdot,\cdot|t,x))_{(t,x)\in [0,T]\times\rd}\) denotes the disintegration of the measure \(\chi^k\) along the first two variables. Thus, for each \(R>0\), the set
	\begin{equation*}\label{construct:intro:Z_R}
	\begin{split}
		Z_R\triangleq \big\{(t,x,&\beta,\Sigma)\in [0,T]\times\rd\times\rd\times\Sd:\\ &|x|\leq R,\, |\beta|\leq C_b^0(1+R^2)^{1/2},\, \|\Sigma\|_F\leq C_A^0(1+R^2)\big\}\end{split}\end{equation*} is compact and \begin{equation*}\label{construct:equality:Z_R} \chi^{(k)}(Z_R)=\mu^{(k)}([0,T]\times \mathbb{B}_R).\end{equation*} Here, \(\mathbb{B}_R\) stands for the closed ball in \(\rd\) of radius 
\(R\) centered at the origin. Thus,  since \(\{\mu^{(k)}\}_{k=1}^\infty\) is tight, \(\{\chi^{(k)}\}_{k=1}^\infty\) is also tight. Hence, by Prokhorov's theorem, the sequence \(\{\chi^{(k)}\}_{k=1}^\infty\) is precompact. Without loss of generality, assume the whole sequence \(\{\chi^{(k)}\}_{k=1}^\infty\) converges to a measure \(\chi\) on \([0,T]\times\rd\times\rd\times\Sd\). Moreover, \(\operatorname{p}^{1,2}\sharp\chi=\mu\). 
	
	Notice that, due to~\eqref{construct:incl:chi_k_tx}, it suffices to consider in~\eqref{construct:equality:chi_k} the integral over the compact set \(Z_R\) for \(R\) such that \(\varphi(t,x)\equiv 0\) for \(|x|>R\). Thus, using the narrow convergence of \(\{\chi^{(k)}\}_{k=1}^\infty\) to \(\chi\) and passing to the limit in~\eqref{construct:equality:chi_k}, we obtain, for every \(\varphi\in C_c^{1,2}((0,T)\times\rd)\), the equality
	\begin{equation}\label{construct:equality:chi_limit}\begin{split}
			\int_{[0,T]\times\rd\times\rd\times \Sd}\bigg[\partial_t\varphi&(t,x)+\nabla\varphi(t,x)\cdot\beta \\&+\frac{1}{2}\operatorname{tr}\big(\nabla^\top\nabla\varphi(t,x)\Sigma\big)\bigg]\chi(d(t,x,\beta,\Sigma))=0.\end{split}\end{equation}  
		
		Define 
	\[b(t,x)\triangleq \int_{\rd\times \Sd}\beta\,\chi(d(\beta,\Sigma)|t,x),\qquad A(t,x)\triangleq \int_{\rd\times \Sd}\Sigma\,\chi(d(\beta,\Sigma)|t,x).\] Then, by the definition of \(\mu\) and~\eqref{construct:equality:chi_limit}, \((m_t)_{t\in [0,T]}\) solves
	\[\partial_tm_t+L_t^*m_t=0,\qquad m_0=\nu\] for \[L_t\varphi(x)=\nabla\varphi(x)\cdot b(t,x)+\frac{1}{2}\operatorname{tr}\big(\nabla^\top\nabla\varphi(t,x)\cdot A(t,x)\big).\] It remains to show that \((b(t,x),A(t,x))\in F(t,x,m_t)\) for \(\mu\)-a.e. \((t,x)\).
	Using \cite[Lemma 5.1.7]{ambrosio} and the lower semicontinuity of \((t,x,\beta,\Sigma)\mapsto \operatorname{dist}((\beta,\Sigma),F^{[l]}(t,x,m_t))\), from~\eqref{construct:equality:F_l_chi_k}, we get
	\begin{equation*}\label{construct:equality:F_l_chi_limit}
		\int_{[0,T]\times\rd\times\rd\times \Sd}\operatorname{dist}((\beta,\Sigma),F^{[l]}(t,x,m_t))\,\chi(d(t,x,\beta,\Sigma))=0.
	\end{equation*}

	 This implies that, for \(\mu\)-a.e. \((t,x)\), \(\chi(\cdot,\cdot|t,x)\)-a.e. \((\beta,\Sigma)\) lies in \(F^{[l]}(t,x,m_t)\). Since \(F^{[l]}(t,x)\) is convex, for \(\mu\)-a.e. \((t,x)\),
	\[(b(t,x),A(t,x))\in F^{[l]}(t,x,m_t).\] Using~\eqref{construct:equality:F_limit}, we obtain \((b(t,x),A(t,x))\in F(t,x,m_t)\) for \(\mu\)-a.e. \((t,x)\).
	
	The fact that \(\mathcal{S}_F(\nu)\subset C([0,T];\prd)\) follows from Lemma~\ref{bound:lm:lip}.
\end{proof}

\section{Compactness of  solutions to the FPK inclusion}\label{sect:limit}

The proof of Theorem~\ref{main:th:compact} relies on the relaxation technique introduced in the definitions below.
\begin{definition}\label{limit:def:chi}
	We say that a measure \(\chi\) on \([0,T]\times\rd\times\rd\times\Sd\) is a relaxed control process for the initial distribution \(\nu\in\prd\) provided  that 
\begin{itemize}
	\item \(\operatorname{p}^1\sharp\chi=\Lebesgue\);
	\item for every \(\varphi\in C^{1,2}_c((0,T)\times\rd)\),
	\[\int_{[0,T]\times\rd}\bigg[\partial_t\varphi(t,x)+\nabla\varphi(t,x)\cdot \beta+\frac{1}{2}\operatorname{tr}( \Sigma\,\nabla^\top\nabla\varphi(t,x))\bigg]\chi(d(t,x,\beta,\Sigma))=0;\]
	\item if \(\mu\triangleq \operatorname{p}^{1,2}\sharp\chi\) and \(m_t\triangleq \mu(\cdot|t)\), then \((m_t)_{t\in [0,T]}\) is continuous, \(m_0=\nu\), and
	\[\int_{[0,T]\times \rd\times\rd\times\Sd}\operatorname{dist}((\beta,\Sigma);F(t,x,m_t))\,\chi(d(t,x,\beta,\Sigma))=0.\]
\end{itemize}
The set of relaxed control processes for the initial distribution \(\nu\in\rd\) is denoted by~\(\mathbb{X}[\nu]\).
\end{definition}
The link between  \(\mathbb{X}[\nu]\) and \(\mathcal{E}_F(\nu)\) is given by the two operation defined as follows. 
\begin{definition}\label{limit:def:i}
Given \(((m_t),b,A)\in \mathcal{E}_F(\nu)\) we define \(\mathscr{i}[(m_t),b,A]\) to be a measure on \([0,T]\times\rd\times\rd\times\Sd\) such that
\begin{equation*}\label{limit:intro:chi_m_t}
	\mathscr{i}[(m_t)_{t\in [0,T]},b,A]\triangleq (\operatorname{Id},b,A)\sharp\mu\end{equation*} for \(
	\mu\triangleq \Lebesgue\star(m_t)_{t\in [0,T]}\).
\end{definition}
Notice that \(\mathscr{i}\) maps \(\mathcal{E}_F(\nu)\) to \(\mathbb{X}[\nu]\).
\begin{definition}
If \(\chi\in\mathbb{X}[\nu]\), then we define 
the mapping \(\mathscr{j}\) by the rule
 \[\mathscr{j}[\chi]\triangleq ((m_t)_{t\in [0,T]},b,A),\] where 
\begin{equation}\label{limit:intro:m}
	\mu\triangleq \operatorname{p}^{1,2}\sharp\chi,\qquad m_t\triangleq \mu(\cdot|t),\end{equation}
\begin{equation*}\label{limit:intro:b_A}b(t,x)\triangleq \int_{\rd\times \Sd}\beta\,\chi(d(\beta,\Sigma)|t,x),\qquad A(t,x)\triangleq \int_{\rd\times \Sd}\Sigma\,\chi(d(\beta,\Sigma)|t,x).\end{equation*} Moreover, if \((m_t)_{t\in [0,T]}\) is defined by~\eqref{limit:intro:m}, we denote it by \(\mathscr{f}[\chi]\).
\end{definition}
Due to Hypothesis~\ref{main:hyp:F_values}, 
\(\mathscr{j}\) maps \(\mathbb{X}[\nu]\) onto \(\mathcal{E}_F(\nu).\)
Furthermore, \(\mathscr{j}\circ\mathscr{i}=\operatorname{Id}\).

\begin{lemma}\label{limit:lm:bound} Let 
	\begin{itemize}
		\item \(\nu\in\prd\) satisfy \(\mathcal{M}_2^2(\nu)\leq \hat{c}\) for some constant \(\hat{c}\);
		\item \(\chi\in\mathbb{X}[\nu]\);
		\item  \((m_t)_{t\in [0,T]}=\mathscr{f}[\chi]\).
		\end{itemize} Then, \(\mathcal{M}_2(m_t)\) is  bounded by a constant \(C_4\) depending only on \(T,C_b,C_A\) and \(\hat{c}\).

\end{lemma}
\begin{proof} Let  \(((m_t)_{t\in [0,T]},b,A)\triangleq \mathscr{j}[\chi]\). This means that  \((m_t)_{t\in [0,T]}\)
		 satisfies the FPK equation
\[\partial_t m_t+L_t^*m_t=0,\]  
 where, given \(t\in [0,T]\), \(L_t\) is a generator defined for \(\phi\in C_b^2(\rd)\) by the rule: \[L_t\phi(x)=\nabla\phi(x)\cdot b(t,x)+\frac{1}{2}\operatorname{tr}\big(\nabla^\top\nabla \phi(x)\, A(t,x)\big).\] Moreover, \(m_0=\nu\). Due to Hypothesis~\ref{main:hyp:F_values}, \[|b(t,x)|\leq C_b(1+|x|+\mathcal{M}_2(m_t)),\ \  \|A(t,x)\|\leq C_A(1+|x|^2+\mathcal{M}_2^2(m_t))\]
for \((\Lebesgue\star(m_t)_{t\in [0,T]})\)-a.e. \((t,x)\).

	Using the Lyapunov function \(V_2(x)\triangleq |x|^2\), similarly to~\eqref{construct:ineq:L_n}, we have 
	\[L_tV_2(x)\leq C'_5(1+|x|^2+\mathcal{M}_2^2(m_t)).\] For some constant \(C'_5\). 
	
	Truncation arguments as in \cite[Theorem 7.1.1]{Bogachev2015} give
	\[\mathcal{M}_2^2(m_t)\leq \hat{c}+C'_5 t+ 2C'_5\int_0^t\mathcal{M}_2^2(m_{t'})\,dt'.\] Gronwall's inequality yields the result.
\end{proof}

\begin{lemma}\label{limit:lm:incl_chi} Assume that \(\mathcal{K}\subset \prd\) is a compact. Then, there exists a constant~\(C_5\) such that, if \(\chi\in \mathbb{X}[\nu]\) for some \(\nu\in\mathcal{K}\), \(\mu\triangleq \operatorname{p}^{1,2}\sharp\chi\), then, for \(\mu\)-a.e. \((t,x)\in [0,T]\times\rd\),
	\begin{equation*}\label{limit:ineq:supp}
		\operatorname{supp}(\chi(\cdot,\cdot|t,x))\subset \big\{(\beta,\Sigma)\in\rd\times\Sd:\, |\beta|\leq C_5\mathscr{a}(x),\, \|\Sigma\|_F\leq C_5\mathscr{a}^2(x)\big\},
	\end{equation*} where \((\chi(\cdot,\cdot|t,x))_{(t,x)\in [0,T]\times\rd}\) denotes the disintegration of the measure~\(\chi\).
\end{lemma}
\begin{proof}
	  First, notice that \(\mathcal{M}_2(\nu)\) is bounded by a constant \(\hat{c}\) determined by the compact \(\mathcal{K}\). Due to Lemma~\ref{limit:lm:bound}, we have that \(\mathcal{M}_2(m_t)\leq C_4\), where   \((m_t)_{t\in [0,T]}\triangleq \mathscr{f}[\chi]\). This, the definition of the set \(\mathbb{X}[\nu]\) and Hypothesis~\ref{main:hyp:F_values} give the conclusion of the lemma. 
\end{proof}

\begin{lemma}\label{limit:lm:chi_m_comp} Given a compact \(\mathcal{K}\subset \prd\) there exists a compact set \(\mathcal{K}'\subset\prd\) such that, if \(\chi\in\mathbb{X}[\nu]\) for some \(\nu\in \mathcal{K}\), while \((m_t)_{t\in [0,T]}=\mathscr{f}[\chi]\), then 
	\(m_t\in\mathcal{K}'\) for every \(t\in [0,T]\). Moreover, the set
	\[\big\{\mathscr{f}[\chi]:\, \chi\in \mathbb{X}[\nu],\, \nu\in \mathcal{K}\big\}\] is precompact in \(C([0,T];\prd)\).
\end{lemma}
\begin{proof}Notice that \(\mathscr{f}[\chi]\in\mathcal{S}_F[\nu]\). This means that \((m_t)_{t\in [0,T]}\) satisfies the FPK equation with the coefficients \(b,A\) such that \(|b(t,x)|\leq C_5\mathscr{a}(x)\) and \(\|A(t,x)\|_F\leq C_5\mathscr{a}^2(x)\) (the latter estimate follows from Lemma~\ref{limit:lm:incl_chi}). Applying Lemma~\ref{bound:lm:growth}, we obtain the first statement. To deduce the second statement, it suffices to additionally apply Lemma~\ref{comapct:lm:limit}.
\end{proof}

\begin{lemma}\label{limit:lm:chi_lim}
	Assume that \(\{\nu^n\}_{n=1}^\infty\subset\prd\) converges to some \(\nu\in\prd\). Additionally, let 
	\begin{itemize}
		\item \(\{\chi^n\}_{n=1}^\infty\) be a sequence of measures on \([0,T]\times\rd\times\rd\times\Sd\) such that each \(\chi^n\in\mathbb{X}[\nu^n]\);
		\item for each \(n\), \((m_t^n)\triangleq \mathscr{f}[\chi^n]\).
	\end{itemize} Then, there exists a subsequence \(\{n_k\}_{k=1}^\infty\), \(\chi\in\mathbb{X}[\nu]\) such that, for \(\chi^{(k)}\triangleq \chi^{n_k}\), and \(m^{(k)}_t\triangleq m_t^{n_k}\) one has that 
	\begin{itemize}
		\item \(\{\chi^{(k)}\}_{k=1}^\infty\) narrowly converges to \(\chi\);
		\item \(\sup_{t\in [0,T]}W_2(m^{(k)}_t,m_t)\rightarrow 0\) as \(k\rightarrow\infty\), where \((m_t)_{t\in [0,T]}=\mathscr{f}[\chi]\).
	\end{itemize} 	
\end{lemma}
\begin{proof}
	For each natural \(n\), we put \(\mu^n\triangleq \operatorname{p}^{1,2}\sharp\chi^n\), \((m^n_t)_{t\in [0,T]}\triangleq \mathscr{f}[\chi^n]\). Lemma~\ref{limit:lm:chi_m_comp} gives that the sequence \(\{(m^n_t)_{t\in [0,T]}\}_{n=1}^\infty\) is precompact. We choose a subsequence \(\{n_k\}_{k=1}^\infty\) and have that \(\sup_{t\in [0,T]}W_2(m^{(k)}_t,m_t)\rightarrow 0\) as \(k\rightarrow\infty\) for some \((m_t)_{t\in [0,T]}\in C([0,T];\prd)\). Furthermore, set \[\mu^{(k)}\triangleq \operatorname{p}^{1,2}\sharp\chi^{(k)}=\Lebesgue\star(m^{(k)})_{t\in [0,T]}.\] Additionally, using Lemma~\ref{compact:lm:equivalence}, we deduce that \(\{\mu^{(k)}\}_{n=1}^\infty\) converges narrowly to \(\mu\triangleq \Lebesgue\star (m_t)_{t\in [0,T]}\). In particular,~\(\{\mu^{(k)}\}_{k=1}^\infty\) is tight. From this and Lemma~\ref{limit:lm:incl_chi}, it directly follows that the sequence~\(\{\chi^{(k)}\}_{k=1}^\infty\) is tight. We assume that this sequence converges narrowly to some measure~\(\chi\) on \([0,T]\times\rd\times\rd\times\Sd\). Moreover, \(\operatorname{p}^{1,2}\sharp\chi=\mu\). 
	
	Recall that, since \(\chi^{(k)}\in\mathbb{X}[\nu^{(k)}]\), we have, for each \(\varphi\in C_c^{1,2}((0,T)\times\rd)\),
	\[\begin{split}
		\int_{[0,T]\times \rd\times\rd\times\Sd}
		\Big[\partial_t \varphi(t,&x)+\nabla\varphi(t,x)\cdot \beta\\&+\frac{1}{2}\operatorname{tr}\big(\nabla^\top\nabla \varphi(t,x)\cdot\Sigma\big)\Big]\chi^{(k)}(d(t,x,\beta,\Sigma))=0.\end{split}\]
	By Lemma~\ref{limit:lm:chi_m_comp}, one can consider the integral above over the compact set
	\[
	\begin{split}
		\big\{(t,x,\beta,\Sigma)\in [0,T]\times&\rd\times\rd\times\Sd:\\ &|x|\leq R,\, |\beta|\leq C_5(1+R^2)^{1/2}, \, \|\Sigma\|_F\leq C_5(1+R^2)\big\},
	\end{split}
	\]
	where \(R\) is chosen so that \(\varphi(t,x)\equiv 0\) whenever \(|x|>R\).
	Passing to the limit gives, for every \(\varphi\in C_c^{1,2}((0,T)\times\rd)\),
	\begin{equation}\label{limit:equality:varphi_limit}\begin{split}
			\int_{[0,T]\times \rd\times\rd\times\Sd}
			\Big[\partial_t \varphi(t,&x)+\nabla\varphi(t,x)\cdot \beta\\&+\frac{1}{2}\operatorname{tr}\big(\nabla^\top\nabla \varphi(t,x)\cdot\Sigma\big)\Big]\chi(d(t,x,\beta,\Sigma))=0.\end{split}\end{equation}
	
	Now, for each natural \(l\), define
	\[\widetilde{F}^{[l]}(t,x,m)\triangleq \overline{\operatorname{co}}\bigcup_{m':\, W_2(m,m')\leq 2^{-l}}F(t,x,m').\] As before,
	\begin{equation}\label{limit:equality:F_F_l}F(t,x,m)=\bigcap_{l=1}^\infty\widetilde{F}^{[l]}(t,x,m).\end{equation} Since the sequence of flows of probabilities \(\{(m^{(k)}_t)_{t\in [0,T]}\}_{n=1}^\infty\) converges to \((m_t)_{t\in [0,T]}\) in \(C([0,T];\prd)\), for each \(l\) and sufficiently large \(k\) depending on \(l\),
	\[F(t,x,m^{(k)}_t)\subset \widetilde{F}^{[l]}(t,x,m_t).\] Thus, for each \(l\), there exists \(k_0\) such that, for all \(k\geq k_0\),
	\[\int_{[0,T]\times\rd\times\rd\times \Sd}\operatorname{dist}((\beta,\Sigma),\widetilde{F}^{[l]}(t,x,m_t))\,\chi^{k}(d(t,x,\beta,\Sigma))=0.\] Passing to the limit and using the lower semicontinuity of the distance function \((t,x,\beta,\Sigma)\mapsto \operatorname{dist}((\beta,\Sigma),\widetilde{F}^{[l]}(t,x,m_t))\) and \cite[Lemma 5.1.7]{ambrosio}, we obtain
	\begin{equation*}\label{limit:equality:dist}\int_{[0,T]\times\rd\times\rd\times \Sd}\operatorname{dist}((\beta,\Sigma),\widetilde{F}^{[l]}(t,x,m_t))\,\chi^\infty(d(t,x,\beta,\Sigma))=0.\end{equation*} Now we use~\eqref{limit:equality:F_F_l}. Thus,
	\begin{equation*}\label{limit:equality:dist}\int_{[0,T]\times\rd\times\rd\times \Sd}\operatorname{dist}((\beta,\Sigma),F(t,x,m_t))\,\chi^\infty(d(t,x,\beta,\Sigma))=0.\end{equation*}
	This together with the equalities~\eqref{limit:equality:varphi_limit} and \(m_0=\nu\) completes the proof.
	
\end{proof}

\begin{proof}[Proof of Theorem~\ref{main:th:compact}] Without loss of generality, assume that, for each~$n$, there exists a measure $\nu^n\in\mathcal{K}$ such that $(m^n_t)_{t\in [0,T]}\in \mathcal{S}_F(\nu^n)$ and $W_2(\nu^n,\nu)\to 0$ as $n\to\infty$. The inclusion $(m^n_t)_{t\in [0,T]}\in \mathcal{S}_F(\nu^n)$ means that 
	$((m^n_t)_{t\in [0,T]},b^n,A^n)\in \mathcal{E}_F(\nu^n)$ for some functions $b^n:[0,T]\times\rd\rightarrow\rd$, $A^n:[0,T]\times\rd\rightarrow\Sd$. Furthermore, given~$n$, we construct a measure~$\chi^n\triangleq \mathscr{i}[(m^n_t)_{t\in [0,T]},b^n,A^n]$. Lemma~\ref{limit:lm:chi_lim} yields that, up to a subsequence, $\{(\chi^n,(m_t^n)_{t\in [0,T]})\}_{n=1}^\infty$ converges to some $(\chi,(m_t)_{t\in [0,T]})$ in $\mathcal{M}([0,T]\times\rd\times\rd\times\Sd)\times C([0,T];\prd)$, while $(m_t)_{t\in [0,T]}=\mathscr{f}[\chi]$. Since $\chi\in\mathbb{X}[\nu]$, $(m_t)_{t\in [0,T]}\in\mathcal{S}_F(\nu)$.
\end{proof}

\section{Analysis of the control problem for FPK inclusion}\label{sect:control} The aim of this section is to prove Theorem~\ref{main:th:control}. 

\begin{proof}[Proof of Theorem~\ref{main:th:control}]

	Consider the following functional for \(\chi\in\mathbb{X}[\nu]\):
	\[I[\chi]\triangleq G(m_T)+\int_{[0,T]\times \rd\times\rd\times\Sd}f(t,x,m_t,\beta,\Sigma)\,\chi(d(t,x,\beta,\Sigma))\] with \(m_t=(\operatorname{p}^{1,2}\sharp\chi)(\cdot|t)\). We claim that \(I\) is real-valued. Indeed, from Lemma~\ref{limit:lm:incl_chi} and the assumptions of the theorem, for \(\chi\)-a.e. \((t,x,\beta,\Sigma)\),
	\[|f(t,x,m_t,\beta,\Sigma)|\leq C_f(m_t)C_5\mathscr{a}^2(x).\] Moreover, using Lemma~\ref{limit:lm:chi_m_comp}, we have that \(m_t\) lies in the compact \(\mathcal{K}_0\) and, since \(C_f\) is bounded on this compact, the integral is finite for each \(\chi\in\mathbb{X}[\nu]\) and uniformly bounded on \(\mathbb{X}[\nu]\). Moreover, \(G\) is bounded below on the compact \(\mathcal{K}_0\). Thus, \(I\) is bounded below on \(\mathbb{X}[\nu]\). 
	
	Let \(\{\nu^n\}_{n=1}^\infty\subset\prd\) converge to some \(\nu\in \prd\) and let a sequence of measures \(\{\chi^n\}_{n=1}^\infty\) be such that \(\chi^n\in\mathbb{X}[\nu^n]\).
	We wish to prove that there exists \(\chi\in \mathbb{X}[\nu]\) such that 
	\begin{equation}\label{control:ineq:liminf}
		\liminf_{n\rightarrow\infty}I[\chi^n]\geq I[\chi].
	\end{equation}
	
Extracting a subsequence if necessary, we may assume that $(I[\chi^n])_{n=1}^\infty$ converges.

For each \(n\), put \(\mu^n\triangleq \operatorname{p}^{1,2}\sharp\chi^n\) and \((m_t^n)_{t\in [0,T]}\triangleq \mathscr{f}[\chi^n]\). 
	
	By Lemma~\ref{limit:lm:chi_m_comp}, the sequence \(\{(\chi^n,(m_t^n)_{t\in [0,T]})\}_{n=1}^\infty\) is precompact in \(\mathcal{M}([0,T]\times\rd\times\rd\times\Sd)\times C([0,T];\prd)\). Relabeling, assume it converges to some pair \((\chi^\infty,(m_t^\infty)_{t\in [0,T]})\), where \(\chi^\infty\in\mathbb{X}[\nu]\) and \((m_t^\infty)_{t\in [0,T]}=\mathscr{f}[\chi^\infty]\). 
	Now we wish to prove that 
	\begin{equation}\label{control:ineq:main}
		\begin{split}	
			\lim_{n\to\infty}\int_{[0,T]\times \rd\times\rd\times\Sd}f(t,x,m^n_t,&\beta,\Sigma)\chi^n(d(t,x,\beta,\Sigma))\\\geq 
			\int_{[0,T]\times \rd\times\rd\times\Sd}f(&t,x,m_t^\infty,\beta,\Sigma)\chi^\infty(d(t,x,\beta,\Sigma)).
	\end{split}\end{equation}
	To this end, we fix \(\varepsilon>0\).
	
	Lemma~\ref{limit:lm:chi_m_comp} applied to the compact \(\mathcal{K}\triangleq \{\nu^n\}_{n=1}^\infty\cup\{\nu\}\) gives that the sets of probability measures \(\{m_t^n:\, t\in [0,T],\, n\in\mathbb{N}\cup\{\infty\}\}\), \(\{m_t:\, t\in [0,T]\}\) lie in a compact \(\mathcal{K}'\). Let \(g\) be the corresponding moment gauge function from Proposition~\ref{gauge:prop:g}. The second moments are uniformly bounded on \(\mathcal{K}'\) (see Lemma~\ref{limit:lm:bound}). Thus, due to Lemma~\ref{limit:lm:incl_chi}, for \(\chi^n\)-a.e. \((t,x,\beta,\Sigma)\),
	\[|f(t,x,m^n_t,\beta,\Sigma)|\leq C'_6\mathscr{a}^2(x)\] for some constant \(C'_6\). 
	Hence,
	\[\int_{[0,T]\times \rd\times\rd\times\Sd}g(|f(t,x,m^n_t,\beta,\Sigma)|^{1/2})\cdot|f(t,x,m^n_t,\beta,\Sigma)|\,\chi^n(d(t,x,\beta,\Sigma))\] is uniformly bounded in \(n\in\mathbb{N}\cup\{\infty\}\). Since \(g\) diverges monotonically,
	\[\int_{|f(t,x,m^n_t,\beta,\Sigma)|\geq N}|f(t,x,m^n_t,\beta,\Sigma)|\,\chi^n(d(t,x,\beta,\Sigma))\rightarrow 0\quad\text{as }N\rightarrow\infty\] uniformly in \(n\in\mathbb{N}\cup\{\infty\}\). Now set
	\[f_N(t,x,\eta,\beta,\Sigma)\triangleq f(t,x,\eta,\beta,\Sigma)\vee (-N)\] and choose \(N_\varepsilon\) such that
	\begin{equation}\label{control:ineq:f_N:eps}
		\int_{[0,T]\times \rd\times\rd\times\Sd}\big[f(t,x,m^n_t,\beta,\Sigma)-f_{N_\varepsilon}(t,x,m^n_t,\beta,\Sigma)\big]\,\chi^n(d(t,x,\beta,\Sigma))\geq -\varepsilon.
	\end{equation} For \(t\in [0,T],x\in\rd,\eta\in\mathcal{K},\beta\in\rd,\Sigma\in\Sd\), define 
	\[\begin{split}
		f_{N_\varepsilon,k}(t,x,\eta,\beta,\Sigma)\triangleq&{}\\ \inf\Bigg\{f_{N_\varepsilon}(t',x',&\eta',\beta',\Sigma')\\+k\big(|&t-t'|+|x-x'|+W_2(\eta,\eta')+|\beta-\beta'|+\|\Sigma-\Sigma'\|_F\big):\\ &{}\hspace{50pt} t'\in [0,T],\, x'\in\rd,\, \eta'\in\mathcal{K},\,\beta'\in\rd,\,\Sigma'\in\Sd \Bigg\}. \end{split}\]
	Then, for every \(\eta\in\mathcal{K}'\),
	\begin{equation}\label{control:ineq:f_N_k}
		\begin{split}
			-N_\varepsilon\leq f_{N_\varepsilon,k}(t,x,\eta,\beta,\Sigma)\leq f_{N_\varepsilon}(t,x,\eta,\beta,\Sigma)\leq |f(t,x,\eta,\beta,\Sigma)|\leq  C'_6\mathscr{a}^2(x).\end{split}\end{equation}  Moreover, \(f_{N_\varepsilon,k}\) increases pointwise to \(f_{N_\varepsilon}\).

	Since the probabilities \(m_t^n\) lie in the compact \(\mathcal{K}'\),   there exists \(R_\varepsilon\) such that 
	\[\int_{|x|>R_\varepsilon}\mathscr{a}^2(x)\mu^n(d(t,x))<\varepsilon/C'_6.\] Here we denote \(\mu^n\triangleq \operatorname{p}^{1,2}\sharp\chi^n\). We put
	\[K_\varepsilon\triangleq \big\{(t,x,\beta,\Sigma):\, t\in [0,T],\, |x|\leq R_{\varepsilon},\, |\beta|\leq  C_5(1+R_{\varepsilon }^2)^{1/2},\, \|\Sigma\|_F\leq  C_5(1+R_{\varepsilon }^2)\big\},\]
	\[\widetilde{K}_\varepsilon\triangleq ([0,T]\times\rd\times\rd\times\Sd)\setminus K_\varepsilon.\] Thanks to Lemma~\ref{limit:lm:incl_chi},
	\(\chi^n(K_\varepsilon)=\mu^n([0,T]\times \{|x|\leq R\})\).  
	
	Thus, estimate~\eqref{control:ineq:f_N_k} implies that, for each \(n\in \mathbb{N}\cup\infty\) and \(k\in\mathbb{N}\),
	\begin{equation}\label{control:ineq_compliment}
		\int_{\widetilde{K}_\varepsilon}|f_{N_\varepsilon,k}(t,x,m^n_t,\beta,\Sigma)|\leq \varepsilon.
	\end{equation} Furthermore, each function \(f_{N_\varepsilon,k}\) is continuous, hence uniformly continuous on the compact set \(\{(t,x,\eta,\beta,\Sigma): (t,x,\beta,\Sigma)\in K_\varepsilon,\, \eta\in\mathcal{K}'\}\). Since \((m_t^n)\) converges uniformly to \((m^\infty_t)\), for sufficiently large \(n\) and all \((t,x,\beta,\Sigma)\in K_\varepsilon\),
	\begin{equation}\label{control:ineq:f_N_k_infi}|f_{N_\varepsilon,k}(t,x,m^n_t,\beta,\Sigma)-f_{N_\varepsilon,k}(t,x,m^\infty_t,\beta,\Sigma)|\leq\varepsilon/T.\end{equation} Hence, for   each \(k\) we have that
	\[\begin{split}
		\lim_{n\rightarrow\infty}&\int_{[0,T]\times \rd\times\rd\times\Sd}f(t,x,m^n_t,\beta,\Sigma)\,\chi^n(d(t,x,\beta,\Sigma))\\&\geq \lim_{n\rightarrow\infty}\int_{[0,T]\times \rd\times\rd\times\Sd}f_{N_\varepsilon}(t,x,m^n_t,\beta,\Sigma)\,\chi^n(d(t,x,\beta,\Sigma))-\varepsilon\\ &\geq 
		\lim_{n\rightarrow\infty}\int_{[0,T]\times \rd\times\rd\times\Sd}f_{N_\varepsilon,k}(t,x,m^n_t,\beta,\Sigma)\,\chi^n(d(t,x,\beta,\Sigma))-\varepsilon \\ &= \lim_{n\rightarrow\infty}\Bigg[\int_{K_\varepsilon}f_{N_\varepsilon,k}(t,x,m^n_t,\beta,\Sigma)\,\chi^n(d(t,x,\beta,\Sigma))\\&\hspace{70pt}+\int_{\widetilde{K}_\varepsilon}f_{N_\varepsilon,k}(t,x,m^n_t,\beta,\Sigma)\,\chi^n(d(t,x,\beta,\Sigma))\Bigg]-\varepsilon \\ &\geq \lim_{n\rightarrow\infty}\int_{K_\varepsilon}f_{N_\varepsilon,k}(t,x,m^\infty_t,\beta,\Sigma)\,\chi^n(d(t,x,\beta,\Sigma))-3\varepsilon
	\end{split} 
	\] In the last inequality, we utilized estimates~\eqref{control:ineq_compliment},~\eqref{control:ineq:f_N_k_infi}.
	Using the narrow convergence of \(\{\chi^n\}_{n=1}^\infty\) to \(\chi^\infty\), we get
	\[
	\begin{split}
		\lim_{n\rightarrow\infty}\int_{[0,T]\times \rd\times\rd\times\Sd}&f(t,x,m^n_t,\beta,\Sigma)\,\chi^n(d(t,x,\beta,\Sigma))
		\\ \geq &\int_{K_\varepsilon}f_{N_\varepsilon,k}(t,x,m^\infty_t,\beta,\Sigma)\,\chi^\infty(d(t,x,\beta,\Sigma))-3\varepsilon.	
	\end{split}
	\] Now we use estimate~\eqref{control:ineq_compliment} for \(n=\infty\) and obtain
	\[
	\begin{split}
		\lim_{n\rightarrow\infty}\int_{[0,T]\times \rd\times\rd\times\Sd}&f(t,x,m^n_t,\beta,\Sigma)\,\chi^n(d(t,x,\beta,\Sigma))
		\\ \geq &\int_{[0,T]\times \rd\times\rd\times\Sd}f_{N_\varepsilon,k}(t,x,m^\infty_t,\beta,\Sigma)\,\chi^\infty(d(t,x,\beta,\Sigma))-4\varepsilon.	
	\end{split}
	\]
	Then, letting  \(k\to\infty\) (using monotone convergence), we deduce
	\[
	\begin{split}
		\lim_{n\rightarrow\infty}\int_{[0,T]\times \rd\times\rd\times\Sd}f(t,x,m^n_t,\beta,&\Sigma)\,\chi^n(d(t,x,\beta,\Sigma))
		\\ \geq \int_{[0,T]\times \rd\times\rd\times\Sd}f_{N_\varepsilon}(t&,x,m^\infty_t,\beta,\Sigma)\,\chi^\infty(d(t,x,\beta,\Sigma))-3\varepsilon.	
	\end{split}
	\] Since \(f_{N_\varepsilon}\ge f\), letting \(\varepsilon\to0\) implies~\eqref{control:ineq:main}. 
	
	Recall now that the function \(G\) is lower semicontinuous, the inclusion \(m^n_t\in\mathcal{K}'\) and the fact that \(C_G\) is bounded below on \(\mathcal{K}'\) (the latter is due to the compactness of \(\mathcal{K}'\)). Thus, we showed the fulfillment of~\eqref{control:ineq:liminf} for \(\chi\triangleq\chi^\infty\).
	
	Now let \(\{\chi^n\}_{n=1}^\infty\) be minimizing for \(I\) over \(\mathbb{X}[\nu]\). Inequality~\eqref{control:ineq:liminf} immediately gives that \(\chi\) minimizes \(I\) over \(\mathbb{X}[\nu]\). Define the triple \(((m_t)_{t\in [0,T]},b,A)\) to be equal to \(\mathscr{j}[\chi]\). Due to the convexity of the function~\(f\),
	\[\operatorname{Val}(\nu)=J[(m_t)_{t\in [0,T]},b,A].\]
	
	The proof that \(\operatorname{Val}\) is lower semicontinuous is analogous. Here we choose, for each~\(n\),  \(((m^n_t),b^n,A^n)\in\mathcal{E}_F(\nu^n)\) such that 
	\[\operatorname{Val}(\nu^n)=J[(m_t^n),b^n,A^n]\] and let \(\chi^n\triangleq \mathscr{i}((m_t^n)_{t\in [0,T]},b^n,A^n)\). Notice that  \(\operatorname{Val}(\nu^n)=I[\chi^n]\).
	 Inequality~\eqref{control:ineq:liminf} gives, for the limiting measure \(\chi\in\mathbb{X}[\nu]\),
	 \[I[\chi]\leq \liminf_{n\rightarrow\infty} \operatorname{Val}(\nu^n).\] Moreover,  \(\operatorname{Val}(\nu)\leq I[\chi]\).
	 Thus,
	\[\lim_{n\to\infty}\operatorname{Val}(\nu^n)\geq \operatorname{Val}[\nu^\infty].\] This completes the proof.
\end{proof}

\begin{acknowledgement} The research was supported by the Russian Science Foundation (project No. 24-11-00123).
\end{acknowledgement}

\bibliography{McKV.bib}

\end{document}